\documentclass[a4paper,12pt]{amsart} 

\usepackage{amssymb,amsmath}

\newtheorem{thm}{Theorem}[section] 
\newtheorem{prop}[thm]{Proposition}
\newtheorem{cor}[thm]{Corollary} 
\newtheorem{lem}[thm]{Lemma}

\theoremstyle{definition} 
 
\newtheorem{rem}[thm]{Remark}

\newtheorem*{ackn}{Acknowledgement}

\numberwithin{equation}{section}

\begin{document}
\baselineskip=16pt

\title{Variation of the modulus of a foliation}
\author{M. Ciska}
\date{}

\begin{abstract}
The $p$--modulus ${\rm mod}_p(\mathcal{F})$ of a foliation $\mathcal{F}$ on a Riemannian manifold $M$ is a generalization of extremal length of plane curves introduced by L. Ahlfors. We study the variation $t\mapsto{\rm mod}_p(\mathcal{F}_t)$ of the modulus. In particular, we consider product of moduli of orthogonal foliations.
\end{abstract}


\address{
Department of Mathematics and Computer Science \endgraf
University of \L\'{o}d\'{z} \endgraf
ul. Banacha 22, 90-238 \L\'{o}d\'{z} \endgraf
Poland
}
\email{mciska@math.uni.lodz.pl}

\maketitle

\section{Introduction}

Modulus, an inverse of extremal length, of a family of a plane curves is a conformal invariant \cite{ab}. The notion of a modulus can be generalized to any family of submanifolds \cite{bf,bl1} and, hence, to a foliation. Roughly speaking, the $p$--modulus ${\rm mod}_p(\mathcal{F})$ of a $k$--dimensional foliation on a $n$--dimensional Riemannian manifold is the infimum of $p$--th norm over all nonnegative, $p$--integrable functions $f$ such that $\int_Lf\geq 1$ for almost every $L\in\mathcal{F}$. If $n=kp$, then the amount ${\rm mod}_p(\mathcal{F})$ is a conformal invariant.

In this paper, we study the variation of a modulus. We generalize the result obtained by Kalina and Pierzchalski \cite{kp} for codimension one foliations given by a submersion. We assume the existence of a function $f_0$ which realizes the $p$--modulus and do not put any requirements on the dimension and codimension of a foliation on a Riemannian manifold $(M,g)$. The methods used here are different than the one used in \cite{kp} and rely on a integral formula obtained by the author in \cite{ci}:
\begin{equation*}
\int_M f_0^{p-1}\varphi\, d\mu_M=\int_M f_0^p\widehat{\varphi}\,d\mu_M,\quad \textrm{where} \quad\widehat{\varphi}(x)=\int_{L_x}\varphi\, d\mu_{L_x}. 
\end{equation*}

The main formula is the following
\begin{equation}\label{eq:variation}
\frac{d}{dt}{\rm mod}_p(\mathcal{F}_t)^p_{t=0}= -p\int_M f_0^{p-1}\left(g(\nabla f_0,X)+ f_0\rm div_{\mathcal{F}_0}X\right)\,d\mu_M.
\end{equation} 
which is valid for all {\it admissible} foliations i.e. foliations satisfying certain assumptions (see Theorem \ref{t_variation}). We show that all foliations given by a submersion are admissible (Theorem \ref{t_admiss}).

Using the formula \eqref{eq:variation} we obtain conditions for a foliation to be a critical point of a variation. We show that foliation is a critical point of a variation if and only if the gradient of extremal function $f_0$ is a combination of mean curvature of a foliation and distribution orthogonal to this foliation (Corollary \ref{c_critpoint}).

\section{Preliminaries}

Let $(M,g)$ be a Riemannian manifold, $\mathcal{F}$ a $k$--dimensional foliation on $M$. Let $\mu_M$ and $\mu_L$ denote Lebesgue measures on $M$ and $L\in\mathcal{F}$, respectively. Fix the coefficient $p>1$ and let $L^p(M)$ be a space of all $p$--integrable functions on $M$ with respect to $\mu_M$ with the norm $\|f\|_p=\left(\int_M |f|^p\,d\mu_M\right)^{\frac{1}{p}}$. Denote by ${\rm adm}_p(\mathcal{F})$ a subfamily of $L^p(M)$ of all nonnegative functions $f$ such that $\int_L f\,d\mu_L\geq 1$ for almost every $L\in\mathcal{F}$. The $p$--{\it modulus} ${\rm mod}_p(\mathcal{F})$ of $\mathcal{F}$ is defined as follows
\begin{equation*}
{\rm mod}_p(\mathcal{F})=\inf_{f\in{\rm adm}_p(\mathcal{F})}\|f\|_p
\end{equation*}
if ${\rm adm}_p(\mathcal{F})\neq\emptyset$ and ${\rm mod}_p(\mathcal{F})=\infty$ otherwise \cite{ci}. Function $f_0\in{\rm adm}_p(\mathcal{F})$ which realizes the modulus i.e.
\begin{equation*}
\|f_0\|_p={\rm mod}_p(\mathcal{F})
\end{equation*}
is called {\it extremal function} for $p$--modulus of $\mathcal{F}$. An extremal function does not exist for any foliation. Namely, we have the following characterization of existence of $f_0$.

\begin{prop}[\cite{ci}]\label{p_extfun}
There exists an extremal function for $p$--modulus of a foliation $\mathcal{F}$ if and only if for any subfamily $\mathcal{L}\subset\mathcal{F}$ such that ${\rm mod}_p(\mathcal{F})=0$ we have $\mu(\bigcup \mathcal{L})=0$.
\end{prop}

\begin{rem}\label{rem1}
Notice that the modulus for any subfamily $\mathcal{L}\subset\mathcal{F}$ is defined in the same way considering functions defined on $M$ not only on $\bigcup\mathcal{L}$.
\end{rem}

The extremal function has the following properties.

\begin{prop}[\cite{ci}]\label{p_propext}
Assume there exists an extremal function $f_0$ for $p$--modulus of $\mathcal{F}$. Then
\begin{enumerate}
\item $\int_L f_0=1$ for almost every leaf $L\in\mathcal{F}$,
\item $f_0>0$,
\item for any $\varphi\in L^p(M)$ we have $\varphi\in L^1(L)$ for almost every leaf $L\in\mathcal{F}$. 
\end{enumerate}
\end{prop}

Assume there exists an extremal function for a $p$--modulus of a foliation $\mathcal{F}$. Then, by Proposition \ref{p_propext}, for any $\varphi\in L^p(M)$ we have $\varphi\in L^1(L)$ for almost every $L\in \mathcal{F}$. Hence, the following function
\begin{equation*}
\widehat{\varphi}(x)=\int_{L_x}\varphi\,d\mu_{L_x},\quad x\in L_x\in\mathcal{F}.
\end{equation*}
is well defined. 

\begin{thm}[\cite{ci}]\label{t_intfor}
Let $f_0$ be an extremal function for $p$--modulus of $\mathcal{F}$. Let $\varphi\in L^p(M)$ be such that ${\rm esssup}|\varphi|<\infty$ and ${\rm esssup}|\widehat{\varphi}|<\infty$. Then
\begin{equation*}
\int_M f_0^{p-1}\varphi\,d\mu_M=\int_M f_0^p\widehat{\varphi}\,d\mu_M.
\end{equation*}
\end{thm}

Consider now a foliation $\mathcal{F}$ given by the level sets of a submersion $\Phi:M\to N$ i.e. $\mathcal{F}=\{\Phi^{-1}(y)\}_{y\in N}$. Decompose the tangent bundle $TM$ into vertical and horizontal distributions
\begin{equation*}
TM=\mathcal{V}\oplus\mathcal{H},\quad\mathcal{V}=\ker\Phi_{\ast},\quad\mathcal{H}=\mathcal{V}^{\bot}, 
\end{equation*} 
where $\bot$ denotes the orthogonal complement with respect to Riemannian metric on $M$. Then $\Phi_{\ast x}:\mathcal{H}_x\to T_{\Phi(x)}N$, $x\in M$, is a linear isomorphism. Let $\Phi_{\ast x}^*:T_{\Phi(x)}N\to\mathcal{H}_x$ be an adjoint linear operator. The Jacobian $J\Phi$ of $\Phi$ is equal
\begin{equation*}
J\Phi(x)=\sqrt{\det(\Phi_{\ast x}\circ\Phi_{\ast x}^*:\mathcal{H}_x\to\mathcal{H}_x)},\quad x\in M.
\end{equation*}

The condition for existence of an extremal function for a foliation given by the level sets of a submersion takes the following form.

\begin{prop}[\cite{ci}]\label{p_extfunsub}
Let $\mathcal{F}$ be a foliation defined by a submersion $\Phi:M\to N$ such that $J\Phi<C$ for some constant $C$. Let $L_x$ denotes the leaf of $\mathcal{F}$ through $x\in M$ and put $\mathcal{F}_{\infty}=\{x\in M:\, \mu_{L_x}(L_x)=\infty \}$. Assume moreover $\mu_M(M)<\infty$. Then, there is an extremal function for $p$--modulus of $\mathcal{F}$ $($for any $p>1)$ if and only if $\mu_M(\mathcal{F}_{\infty})=0$.
\end{prop}  

There is an explicit formula for an extremal function in the case of a foliation given by a submersion. 

\begin{prop}[\cite{kp,ci}]\label{p_formulaextfun}
If $f_0$ is an extremal function for $p$--modulus of a foliation $\mathcal{F}$ given by the level sets of a submersion $\Phi:M\to N$, then
\begin{equation*}
f_0=\frac{(J\Phi)^{\frac{1}{p-1}}}{\widehat{(J\Phi)^{\frac{1}{p-1}}}}.
\end{equation*}
\end{prop}

\section{Admissible foliations}

Let $(M,g)$ be a Riemannian manifold, $p>1$. Let $X$ be a compactly supported vector field on $M$, $\varphi_t$ a flow of $X$. Let $\mathcal{F}$ be a foliation on $M$ and put $\mathcal{F}_t=\varphi_t(\mathcal{F})$.

We say that $X$ is {\it admissible} for $p$--modulus of $\mathcal{F}$ if, for some interval $I=(-\varepsilon,\varepsilon)$, we have
\begin{enumerate} 
\item[(A1)] there exists an extremal function $f_t$ for $p$--modulus of $\mathcal{F}_t$ for all $t\in I$,
\item[(A2)] the function $\alpha(x,t)=(f_t\circ\varphi_t)(x)$ is $C^1$--smooth with respect to variable $t\in I$,
\item[(A3)] there is $h_1\in L^p(M)$ such that $|\alpha(x,t)|<h_1(x)$ for all $t\in I$,
\item[(A4)] there is $h_2\in L^p(M)$ such that $|\frac{\partial\alpha}{\partial t}(x,t)|<h_2(x)$ for all $t\in I$.
\end{enumerate}

In addition, if every compactly supported vector field is admissible for $p$--modulus of $\mathcal{F}$, then we say that $\mathcal{F}$ is $p$--{\it admissible}.

The main result of this section is the following.

\begin{thm}\label{t_admiss}
Let $\mathcal{F}$ be a foliation on a Riemannian manifold $M$ given by the level sets of a submersion $\Phi:M\to N$. Assume 
\begin{enumerate}
\item $C_1<J\Phi<C_2$ and $\hat 1<C_3$ for some positive constants $C_1,C_2,C_3$, \item an extremal function for $p$--modulus of $\mathcal{F}$ is smooth $(p>1)$. 
\end{enumerate}
Then $\mathcal{F}$ is $p$--admissible.
\end{thm}

\begin{proof}
Let $X$ be compactly supported vector field on $M$ and let $\varphi_t$ be a flow of $X$. Put $\mathcal{F}_t=\varphi_t(\mathcal{F})$. Then $\mathcal{F}_t$ is given by the level sets of the submersion $\Phi_t=\Phi\circ\varphi_t^{-1}:M\to N$. Moreover
\begin{equation*}
J\Phi_t=J\Phi\cdot J^{\bot}\varphi_t^{-1},
\end{equation*}
where $J^{\bot}\varphi_t^{-1}$ is a smooth function depending only on differential $\varphi_{t\ast}$ of the map $\varphi_t$. 

Let $L^t_z$ denotes the leaf of $\mathcal{F}_t$ through $z\in M$ i.e.
\begin{equation*}
L^t_z=\Phi_t^{-1}(\Phi_t(z)),\quad z\in M.
\end{equation*}
We divide the proof into few steps.

{\bf Step 1} -- there exist an extremal function $f_t$ for $p$--modulus of $\mathcal{F}_t$.\\ 
Since $C_1<J\Phi<C_2$ and $\varphi_t$ is a flow of compactly supported vector field then the Jacobian $J\Phi_t$ is bounded. Since $\mu_M(\mathcal{F}_{\infty})=0$, then $\mu_M((\mathcal{F}_t)_{\infty})=0$, hence, by Proposition \ref{p_extfunsub}, there exists an extremal function $f_t$ for $p$--modulus of $\mathcal{F}_t$. By Proposition \ref{p_formulaextfun} we have
\begin{gather}\label{eq_funeksft}
\begin{split}
f_t(z) &=\frac{(J\Phi_t)^{\frac{1}{p-1}}(z)}{\int_{L^t_z}(J\Phi_t)^{\frac{1}{p-1}}\,d\mu_{L^t_z}}\\
&=\frac{(J\Phi\circ \varphi_t^{-1}(z))^{\frac{1}{p-1}}(J^{\bot}\varphi_t^{-1}(z))^{\frac{1}{p-1}}}
{\int_{L^0_{\varphi_t^{-1}(z)}}(J\Phi)^{\frac{1}{p-1}}((J^{\bot}\varphi_t^{-1})\circ\varphi_t)^{\frac{1}{p-1}}J^{\top}\varphi_t \,d\mu_{L^0_{\varphi_t^{-1}(z)}}}.
\end{split}
\end{gather}

{\bf Step 2} -- function $t\mapsto \alpha(x,t)=(f_t\circ\varphi_t)(x)$ is $C^1$--smooth.\\ 
By \eqref{eq_funeksft} 
\begin{equation}\label{eq_ftfit}
(f_t\circ\varphi_t)(x)=\frac{(J\Phi(x))^{\frac{1}{p-1}}(J^{\bot}\varphi_t^{-1}(\varphi_t(x)))^{\frac{1}{p-1}}}
{\int_{L^0_x}(J\Phi)^{\frac{1}{p-1}}((J^{\bot}\varphi_t^{-1})\circ\varphi_t)^{\frac{1}{p-1}}J^{\top}\varphi_t \,d\mu_{L^0_x}}
\end{equation}
Functions
\begin{equation*}
t\mapsto\beta_1(x,t)=(J^{\bot}\varphi_t^{-1}(\varphi_t(x)))^{\frac{1}{p-1}}
\end{equation*}
and
\begin{equation*}
t\mapsto\beta_2(x,t)=((J^{\bot}\varphi_t^{-1})\circ\varphi_t)^{\frac{1}{p-1}}J^{\top}\varphi_t
\end{equation*}
are smooth and positive. Since $X$ is compactly supported, for any closed interval $I$ containing $0\in\mathbb{R}$, functions
\begin{equation*}
\beta_1(x,t),\quad \frac{\partial\beta_1}{\partial t}(x,t),\quad \beta_2(x,t),\quad \frac{\partial\beta_2}{\partial t}(x,t),\quad (x,t)\in M\times I
\end{equation*}
are bounded. By Lebesgue dominated convergence theorem, function $I\ni t\mapsto\alpha(x,t)$ is differentiable. Analogously, we show that this function is twice differentiable, hence is $C^1$--smooth.

{\bf Step 3} -- $|f_t\circ\varphi_t|<Cf_0$ and $|\frac{d}{dt}(f_t\circ\varphi_t)|<Cf_0$ for some $C>0$.\
By \eqref{eq_ftfit}
\begin{equation*}
|(f_t\circ\varphi_t)(x)|=\left|\frac{(J\Phi)^{\frac{1}{p-1}}\beta_1(x,t)}
{\int_{L_0^x}(J\Phi)^{\frac{1}{p-1}}\beta_2(x,t)\,d\mu_{L^x_0}}\right|\leq C\frac{(J\Phi)^{\frac{1}{p-1}}}{\int_{L^0_x}(J\Phi)^{\frac{1}{p-1}}d\mu_{L^0_x}}= Cf_0
\end{equation*}
and
\begin{align*}
\left|\frac{d}{dt}\left( f_t\circ\phi_t \right)\right| &=\left| \frac{(J\Phi)^{\frac{1}{p-1}}\frac{\partial\beta_1}{\partial t}(x,t)}
{\int_{L^0_x}(J\Phi)^{\frac{1}{p-1}}\beta_2(x,t)\, d\mu_{L^0_x}}\right.\\
&-\left.\frac{(J\Phi)^{\frac{1}{p-1}}\beta_1(x,t)\cdot \int_{L^0_x}(J\Phi)^{\frac{1}{p-1}}\frac{\partial\beta_2}{\partial t}(x,t)\, d\mu_{L^0_x}}{\left(\int_{L^0_x}(J\Phi)^{\frac{1}{p-1}}\beta_2(x,t)\, d\mu_{L^0_x}\right)^2} \right|\\
&\leq C'\left|\frac{(J\Phi)^{\frac{1}{p-1}}}{\int_{L^0_x}(J\Phi)^{\frac{1}{p-1}}d\mu_{L^0_x}}\right|+
C''\left| \frac{(J\Phi)^{\frac{1}{p-1}}\int_{L^0_x}(J\Phi)^{\frac{1}{p-1}}d\mu_{L^0_x}}{\left(\int_{L^0_x}(J\Phi)^{\frac{1}{p-1}}d\mu_{L^0_x}\right)^2} \right|\\
&= C f_0
\end{align*}

{\bf Step 4} -- $\inf f_0>0$.\\
Follows from the fact that $C_1<J\Phi<C_2$ and $\hat 1<C_3$ i.e.
\begin{equation*}
f_0\geq \frac{C_1^{\frac{1}{p-1}}}{C_2^{\frac{1}{p-1}}C_3}>0.
\end{equation*} 
\end{proof}

By above theorem we get immediately the following corollary.

\begin{cor}\label{c_admiss}
Let $\mathcal{F}$ be a foliation given by the level sets of a submersion $\Phi:M\to N$, where $M$ is compact. Assume an extremal function for $p$--modulus of $\mathcal{F}$ is smooth. Then $\mathcal{F}$ is $p$--admissible.
\end{cor}

\section{Variation of modulus}

In this section, we consider the variation of $p$--modulus under the flow of compactly supported vector field. The formula for a variation implies some results about an extremal function.

Let $\mathcal{F}$ be a $k$--dimensional foliation on a Riemannian manifold $(M,g)$. Denote by ${\rm div}_{\mathcal{F}}X$ the divergence of a vector field $X\in\Gamma(TM)$ with respect to leaves of $\mathcal{F}$ i.e.
\begin{equation*}
{\rm div}_{\mathcal{F}}X=\sum_{i=1}^k g(\nabla_{e_i}X,e_i),
\end{equation*}    
where $e_1,\ldots,e_k$ is a local orthonormal basis of $T\mathcal{F}$ and $\nabla$ the Levi--Civita connection on $M$. Let $H_{\mathcal{F}}$ and $H_{\mathcal{F}^{\bot}}$ denote the mean curvatures of $\mathcal{F}$ and the distribution $\mathcal{F}^{\bot}$ orthogonal to $\mathcal{F}$, respectively. If $X$ is tangent to $\mathcal{F}$ then the divergence ${\rm div}_MX$ on $M$ and the divergence on the leaves ${\rm div}_{\mathcal{F}}X$ are related as follows
\begin{equation}\label{eq_divmf}
{\rm div}_MX={\rm div}_{\mathcal{F}}X-g(X,H_{\mathcal{F}^{\bot}}),\quad X\in\Gamma(T\mathcal{F}).
\end{equation}

Moreover, if $\varphi_t$ is a flow of a vector field $X\in\Gamma(TM)$ and $\mathcal{F}_t=\varphi_t(\mathcal{F})$, then
\begin{equation}\label{eq_divflow}
\frac{d}{dt}J^{\top}\varphi_t(x)_{t=t_0}=J^{\top}\varphi_{t_0}(x){\rm div}_{\mathcal{F}_{t_0}}X,
\end{equation}
where $J^{\top}\varphi_t$ is the Jacobian of $\varphi_t$ restricted to leaves of $\mathcal{F}$. In particular,
\begin{equation*}
\frac{d}{dt}J^{\top}\varphi_t(x)_{t=0}={\rm div}_{\mathcal{F}}X.
\end{equation*}

\begin{thm}\label{t_variation}
Let $\mathcal{F}$ be a foliation on a Riemannian manifold $(M,g)$. Let $X$ be compactly supported vector field on $M$, which is admissible for $p$--modulus of $\mathcal{F}$. Assume there exists smooth extremal function $f_0$ for $p$--modulus of $\mathcal{F}$. Then, the following formula holds
\begin{equation}\label{eq_variation}
\frac{d}{dt}{\rm mod}_p(\mathcal{F}_t)^p_{t=0}= -p\int_M f_0^{p-1}\left(g(\nabla f_0,X)+ f_0\rm div_{\mathcal{F}}X\right)d\mu_M.
\end{equation} 
\end{thm} 

Before we prove above theorem we will need the following technical lemma.
\begin{lem}\label{l_technical}
Let $h:M\times\mathbb{R}^2\to\mathbb{R}$ be defined as follows $h(x,t,s)=(f_s\circ\varphi_t)(x)$.
Then
\begin{equation*}
\frac{\partial h}{\partial x}(x,t,s)=g(\nabla f_s(\varphi_t(x)),X_{\varphi_t(x)}),\quad \frac{\partial h}{\partial s}(x,t,s)=\frac{df_s}{ds}(\varphi_t(x)).
\end{equation*}
In particular,
\begin{equation}\label{eq_derivativeftfit}
\frac{d}{dt}(f_t\circ\varphi_t)(x)=g(\nabla f_t(\varphi_t(x)),X_{\varphi_t(x)})+\frac{df_t}{dt}(\varphi_t(x)).
\end{equation}
\end{lem}
\begin{proof}
Consider two maps $F:M\times \mathbb{R}\to\mathbb{R}$ and $\varphi:M\times\mathbb{R}\to\mathbb{R}$,
\begin{equation*}
F(x,s)=f_s(x),\quad \varphi(x,t)=\varphi_t(x).
\end{equation*}
Then $h=F\circ(\varphi,{\rm id}_{\mathbb{R}})$. Hence
\begin{equation*}
\frac{\partial h}{\partial t}=h_{\ast}\frac{d}{dt}=F_{\ast}(\varphi_{\ast}\frac{d}{dt},0)=F_{\ast}(X,0)=f_{s\ast}X=g(\nabla f_s,X)
\end{equation*}
and
\begin{equation*}
\frac{\partial h}{\partial t}=h_{\ast}\frac{d}{ds}=F_{\ast}(0,\frac{d}{ds})=\frac{df_s}{ds}.
\end{equation*}
\end{proof}

\begin{proof}[Proof of Theorem \ref{t_variation}] 
Since $X$ is admissible for $p$--modulus of $\mathcal{F}$, then there are functions $h_1,h_2\in L^p(M)$ such that conditions (A1) and (A2) hold. By Lemma \ref{l_technical} and \eqref{eq_divflow} we have
\begin{align*}
|\frac{d}{dt}\left(\left(f_t\circ \varphi_t\right)J^\top\varphi_t \right)| &=|\frac{d}{dt}\left(f_t\circ \varphi_t\right)J^{\top}\varphi_t+\left(f_t\circ\varphi_t\right)J^{\top}\varphi_t\rm div_{\mathcal{F}}X|\\
&\leq |\frac{d}{dt}\left(f_t\circ \varphi_t\right)||J^{\top}\varphi_t|+|f_t\circ\varphi_t||J^{\top}\varphi_t||\rm div_{\mathcal{F}}X|\\
&\leq C_1h_2+C_1C_2h_1,
\end{align*}
where $J^{\top}\varphi_t<C_1$, $t\in I$, and ${\rm div}_{\mathcal{F}}X<C_2$. Function $h=C_1h_2+C_1C_2h_1$ is in $L^p(M)$, hence, the existence of extremal function $f_0$ implies that $h\in L^1(L)$ for almost every leaf $L\in \mathcal{F}$ (Proposition \ref{p_propext}). By Lebesgue dominated convergence theorem, Lemma \ref{l_technical} and \eqref{eq_divflow} for any $L\in\mathcal{F}$ we have
\begin{align*}
\frac{d}{dt}\left(\int_L (f_t\circ\varphi_t)J^{\top}\varphi_t\,d\mu_L\right)_{t=0} &=\int_L \frac{d}{dt}((f_t\circ\varphi_t)J^{\top}\varphi_t)_{t=0}\,d\mu_L \\
&=\int_L (g(\nabla f_0,X)+(\frac{df_t}{dt})_{t=0}+f_0{\rm div}_{\mathcal{F}}X)\,d\mu_L.
\end{align*}
Let $L_t$ denotes the leaf of $\mathcal{F}_t$ i.e. $L_t=\varphi_t(L)$, $L\in\mathcal{F}$. Since, by Proposition \ref{p_propext}, $\int_{L_t}f_t\,d\mu_{L_t}=1$, then
\begin{equation*}
0=\frac{d}{dt}\left( \int_{L_t}f_t\,d\mu_{L_t} \right)_{t=0}=\frac{d}{dt}\left( \int_L (f_t\circ\varphi_t)J^{\top}\varphi_t\, d\mu_L \right)_{t=0}.
\end{equation*}
Hence, by above,
\begin{equation*}
\int_L (\frac{df_t}{dt})_{t=0}\, d\mu_L=-\int_L (g(\nabla f_0,X)+f_0{\rm div}_{\mathcal{F}}X)\,d\mu_L.
\end{equation*}
Notice that
\begin{align*}
\int_L (\frac{df_t}{dt})_{t=0}\,d\mu_L &=\int_L \frac{d}{dt}(f_t\circ\varphi_t\circ\varphi^{-1}_t)_{t=0}\,d\mu_L \\
&=\int_L\frac{d}{dt}(f_t\circ\varphi_t)\frac{d}{dt}(\varphi_t^{-1})_{t=0}\,d\mu_L \\
&\leq C\int_L h_1\,d\mu_L
\end{align*}
for some constant $C>0$. Since ${\rm esssup}|\widehat{h_1}|<\infty$, it follows that \begin{equation*}
{\rm esssup}|\widehat{(\frac{df_t}{dt})_{t=0}}|<\infty.
\end{equation*} 
Hence we may use Theorem \ref{t_intfor} for $\varphi=(\frac{df_t}{dt})_{t=0}$. Moreover, we have
\begin{equation*}
\int_M f_t^p\,d\mu_M=\int_{\varphi_t(M)}f_t^p\,d\mu_M=\int_M (f_t\circ\varphi_t)^pJ\varphi_t\,d\mu_M
\end{equation*}
and
\begin{align*}
|\frac{d}{dt}((f_t\circ\varphi_t)^pJ\varphi_t)| &=|p(f_t\circ\varphi_t)^{p-1}\frac{d}{dt}(f_t\circ\varphi_t)J\varphi_t+(f_t\circ\varphi_t)^p\frac{d}{dt}(J\varphi_t)|\\
&\leq C'h_1^{p-1}h_2+C''h_1^p
\end{align*}
for some constants $C',C''>0$. By H\" older inequality function $h_1^{p-1}h_2$ is integrable on $M$, hence function $C'h_1^{p-1}h_2+C''h_1^p$ is integrable on $M$. Thus, by Lebesgue dominated convergence theorem and Lemma \ref{l_technical}
\begin{align*}
\frac{d}{dt}{\rm mod}_p^p(\mathcal{F}_t)_{t=0}&=\frac{d}{dt}\left(\int_M f_t^p\,d\mu_M\right)_{t=0}\\
&=\frac{d}{dt}\left( \int_M (f_t\circ \varphi_t)^pJ\varphi_t\,d\mu_L \right)_{t=0}\\
&=\int_M\frac{d}{dt}( (f_t\circ\varphi_t)^pJ\varphi_t)_{t=0}\,d\mu_M\\
&=\int_M( pf_0^{p-1}\frac{d}{dt}(f_t\circ\varphi_t)_{t=0}+f_0^p{\rm div} X )\,d\mu_M\\
&=\int_M ( pf_0^{p-1}( g(\nabla f_0,X)+(\frac{df_t}{dt})_{t=0} )+f_0^p{\rm div}X )\,d\mu_M\\
&=\int_M(pf_0^{p-1}g(\nabla f_0,X)+f_0^p{\rm div}X)\,d\mu_M+p\int_M f_0^{p-1}(\frac{df_t}{dt})_{t=0}\,d\mu_M\\
&=\int_M {\rm div}(f_0^p X)+p\int_M f_0^{p-1}(\frac{df_t}{dt})_{t=0}\,d\mu_M\\
&=p\int_M f_0^{p-1}(\frac{df_t}{dt})_{t=0}\,d\mu_M.
\end{align*}
Finally, by Theorem \ref{t_intfor}, we obtain
\begin{align*}
\frac{d}{dt}{\rm mod}_p^p(\mathcal{F}_t)_{t=0}&=p\int_M f_0^{p}\widehat{(\frac{df_t}{dt})_{t=0}}\,d\mu_M\\
&=-p\int_M f_0^p(\widehat{g(\nabla  f_0,X)+ f_0\rm {div}_{\mathcal{F}}X})\,d\mu_M\\
&=-p\int_M f_0^{p-1}(g(\nabla f_0,X)+ f_0{\rm div}_{\mathcal{F}}X)\,d\mu_M. 
\end{align*}
\end{proof}

Variation of modulus implies the condition for tangent gradient of an extremal function (compare Corollary 4.4 \cite{ci})

\begin{cor}\label{c_nablabotextfun}
Let $\mathcal{F}$ be a foliation on a Riemannian manifold $(M,g)$. Assume all compactly supported vector fields $X$ tangent to $\mathcal{F}$ are admissible for $p$--modulus of $\mathcal{F}$. Then
\begin{equation*}
\nabla^{\top}(\log f_0)=\frac{1}{p-1}H_{\mathcal{F}^{\bot}}.
\end{equation*} 
\end{cor}

\begin{proof}
Let $X\in\Gamma(T\mathcal{F})$ be compactly supported. Then the flow $\varphi_t$ of $X$ maps $\mathcal{F}$ to $\mathcal{F}$, hence $\mathcal{F}_t=\mathcal{F}$ for all $t$. Thus, by Theorem \ref{t_variation} and formula \ref{eq_divmf}, we have
\begin{align*}
0&=-p\int_M f_0^{p-1}(g(\nabla f_0,X)+ f_0 {\rm div}_{\mathcal{F}}X)\,d\mu_M\\
&=-p\int_M f_0^{p-1}(g(\nabla  f_0,X)+ f_0({\rm div}_MX+g(H_{\mathcal{F}^\bot},X)))\,d\mu_M\\
&=\int_M(-pf_0^{p-1}g(\nabla  f_0,X)-pf_0^p({\rm div}_MX+g(H_{\mathcal{F}^\bot},X)))\,d\mu_M\\
&=\int_M (-g(\nabla  f_0^p,X)-p({\rm div}_M(f_0^{p}X)-g(\nabla f_0^p,X))-p  f_0^p g(H_{\mathcal{F}^\bot},X))\,d\mu_M\\
&=\int_M(g(p\nabla  f_0^p,X)- g(\nabla f_0^p,X)-p f_0^p g(H_{\mathcal{F}^\bot},X))\,d\mu_M.
\end{align*}
Therefore, for compactly supported vector field $X\in\Gamma(T\mathcal{F})$
\begin{equation*}
0=\frac{d}{dt}{\rm mod}_p(\mathcal{F}_t)^p_{t=0}=p\int_M f_0^{p-1}g((p-1)\nabla  f_0- f_0 H_{\mathcal{F}^\bot},X)\,d\mu_M.
\end{equation*}
Since $X\in\Gamma(T\mathcal{F})$ is arbitrary, it follows that
\begin{equation*}
(p-1)\nabla^{\top} f_0- f_0 H_{\mathcal{F}^\bot}=0.
\end{equation*}
\end{proof}

Let $\mathcal{F}$ be a $p$--admissible foliation on a Riemannian manifold $M$. We say that $\mathcal{F}$ is a {\it critical point} of a functional 
\begin{equation}\label{eq_functional}
\mathcal{F}\mapsto{\rm mod}_p(\mathcal{F}),
\end{equation}
if for any compactly supported vector field $X$ we have $\frac{d}{dt}{\rm mod}_p(\mathcal{F}_t)=0$, where $\mathcal{F}_t=\varphi_t(\mathcal{F})$ and $\varphi_t$ is a flow of $X$. 

The following result gives the characterization of critical points.
\begin{cor}\label{c_critpoint}
Let $\mathcal{F}$ be a $p$--admissible foliation on a Riemannian manifold $M$. Then, $\mathcal{F}$ is a critical point of \eqref{eq_functional} if and only if
\begin{equation}\label{eq_ctiticaliff}
\nabla(\log f_0^p)=pH_{\mathcal{F}}+qH_{\mathcal{F}^{\bot}}.
\end{equation}
where $f_0$ is an extremal function for $p$--modulus of $\mathcal{F}$ and $p,q$ are conjugate coefficients i.e. $\frac{1}{p}+\frac{1}{q}=1$.
\end{cor}

\begin{proof}
We begin proof by stating general facts. By Corollary \ref{c_nablabotextfun}
\begin{equation}\label{eq_nablatopfp}
\nabla^{\top}(\log f_0^p)=\frac{p}{p-1}H_{\mathcal{F}^{\bot}}=qH_{\mathcal{F}^{\bot}}.
\end{equation}
Moreover, since ${\rm div}_{\mathcal{F}}X=g(X,H_{\mathcal{F}})$ for $X\in\Gamma(T^{\bot}\mathcal{F})$, then
\begin{equation}\label{eq_varxbot}
\frac{d}{dt}{\rm mod}_p(\mathcal{F}_t)_{t=0}^p=-p\int_M g(\nabla  f_0^p- f_0^pH_{\mathcal{F}},X)\,d\mu_M
\end{equation}
for all compactly supported vector fields $X\in\Gamma(T^{\bot}\mathcal{F})$.

Assume $\mathcal{F}$ is a critical point of a functional \eqref{eq_functional}. Then by \eqref{eq_varxbot} we get $\nabla^{\bot}f_0^p=f_0^pH_{\mathcal{F}}$, hence
\begin{equation*}
\nabla^{\bot}(\log f_0^p)=pH_{\mathcal{F}}.
\end{equation*}
This, together with \eqref{eq_nablatopfp}, implies \eqref{eq_ctiticaliff}.

Assume now \eqref{eq_ctiticaliff} holds. Right--hand side of \eqref{eq_variation} is linear with respect to $X$. Moreover, for $X$ tangent to $\mathcal{F}$ we have (compare proof of Corollary \ref{c_nablabotextfun}) $\frac{d}{dt}{\rm mod}_p(\mathcal{F}_t)_{t=0}=0$. Hence, by \eqref{eq_varxbot}, it suffices to show that
\begin{equation*}
\int_M g(\nabla  f_0^p- f_0^pH_{\mathcal{F}},X)\,d\mu_M=0
\end{equation*}
for all compactly supported vector fields $X\in\Gamma(T^{\bot}\mathcal{F})$. This follows by assumption \eqref{eq_ctiticaliff}, which implies $\nabla^{\bot}(f_0^p)=f_0^pH_{\mathcal{F}}$.
\end{proof}

Now, we consider the case of two orthogonal foliations i.e. we assume that for a given foliation $\mathcal{F}$ on a Riemannian manifold $(M,g)$ the distribution $\mathcal{G}=\mathcal{F}^{\bot}$ is integrable. Let $p,q>1$, be conjugate coefficients i.e. $\frac{1}{p}+\frac{1}{q}=1$.

\begin{thm}\label{t_twofol}
Assume $\mathcal{F}$ is $p$--admissible and $\mathcal{G}$ is $q$--admissible. Then the following conditions are equivalent:
\begin{enumerate}
\item $\mathcal{F}$ and $\mathcal{G}$ are a critical points of the functionals
\begin{equation*}
\mathcal{F}\mapsto{\rm mod}_p(\mathcal{F})\quad\textrm{and}\quad\mathcal{G}\mapsto{\rm mod}_q(\mathcal{G}),
\end{equation*}
respectively,
\item the pair $(\mathcal{F},\mathcal{G})$ is a critical point of a functional
\begin{equation*}
(\mathcal{F},\mathcal{G})\mapsto{\rm mod}_p(\mathcal{F}){\rm mod}_q(\mathcal{G}),
\end{equation*}
\item the extremal functions $f_0$ of $p$--modulus of $\mathcal{F}$ and $g_0$ of $q$--modulus of $\mathcal{G}$ are related as follows
\begin{equation}\label{eq_fgdependent}
{\rm mod}_q(\mathcal{G})^q\cdot f_0^p={\rm mod}_p(\mathcal{F})^p\cdot g_0^q.
\end{equation}
\end{enumerate}
\end{thm}

\begin{proof}
(1) $\Rightarrow$ (2) Follows from the equality
\begin{multline}\label{eq_modpq}
\frac{d}{dt}({\rm mod}_p(\mathcal{F}_t){\rm mod}_q(\mathcal{G}_t))_{t=0}=\\
\frac{d}{dt}{\rm mod}_p(\mathcal{F}_t)_{t=0}\cdot {\rm mod}_q(\mathcal{G})+{\rm mod}_p(\mathcal{F})\cdot \frac{d}{dt}{\rm mod}_q(\mathcal{G}_t)_{t=0}.
\end{multline}
(2) $\Rightarrow$ (1) By existence of extremal functions of $p$--modulus of $\mathcal{F}$ and $q$--modulus of $\mathcal{G}$, it follow by Proposition \ref{p_extfun} that $p$--modulus of $\mathcal{F}$ and $q$--modulus of $\mathcal{G}$ are positive. If $X\in\Gamma(T^{\bot}\mathcal{F})$, then $\frac{d}{dt}{\rm mod}_q(\mathcal{G}_t)_{t=0}=0$, hence, by \eqref{eq_modpq}, 
\begin{equation*}
\frac{d}{dt}{\rm mod}_p(\mathcal{F}_t)_{t=0}=0.
\end{equation*}
If $X\in\Gamma(T\mathcal{F})$, then ${\rm mod}_p(\mathcal{F}_t)_{t=0}=0$. Since the variation of modulus is linear with respect to $X$, it follows that $\frac{d}{dt}{\rm mod}_p(\mathcal{F}_t)_{t=0}=0$ for any compactly supported vector field $X$. Analogously $\frac{d}{dt}{\rm mod}_q(\mathcal{G}_t)_{t=0}=0$ for any compactly supported vector field $X$.\\
(1) $\Leftrightarrow$ (3) By Corollary \ref{c_critpoint} condition (1) is equivalent to the following
\begin{equation}\label{eq_fgnabla}
\nabla(\log f_0^p)=pH_{\mathcal{F}}+qH_{\mathcal{G}}=\nabla(\log g_0^q).
\end{equation}  

Assume (1) holds. Then by \eqref{eq_fgnabla}, $f_0^p=Cg_0^q$ for some constant $C>0$. Hence $f_0$ and $g_0$ are H\" older dependent. Thus
\begin{align*}
{\rm mod}_p(\mathcal{F}){\rm mod}_q(\mathcal{G})&=\left(\int_M f_0^p d\mu_M \right)^{\frac{1}{p}}\left(\int_M g_0^q \,d\mu_M \right)^{\frac{1}{q}}=\int_M f_0 g_0 \,d\mu_M\\
&=\int_M C^{\frac{1}{p}}g_0^{\frac{q}{p}+1} \,d\mu_M=C^{\frac{1}{p}}\int_M g_0^q\, d\mu_M\\
&=C^{\frac{1}{p}}{\rm mod}_q(\mathcal{G}_0)^q.
\end{align*}
Therefore
\begin{equation*}
C=\frac{{\rm mod}_p(\mathcal{F})^p}{{\rm mod}_q(\mathcal{G})^q},
\end{equation*}
so \eqref{eq_fgdependent} holds.

Assume now $f_0$ and $g_0$ are H\" older dependent and \eqref{eq_fgdependent} holds. Thus 
\begin{equation}\label{eq_fgdep2}
\nabla(\log f_0^p)=\nabla(\log g_0^q). 
\end{equation}
By Corollary \ref{c_nablabotextfun} we have
\begin{equation*}
\nabla^{\mathcal{F}}(\log f_0^p)=q H_{\mathcal{G}}\quad\textrm{and}\quad
\nabla^{\mathcal{G}}(\log g_0^q)=p H_{\mathcal{F}},
\end{equation*}
where $\nabla^{\mathcal{F}}$ and $\nabla^{\mathcal{G}}$ denote tangent to $\mathcal{F}$ and to $\mathcal{G}$ part of the gradient, respectively. By above equalities and by \eqref{eq_fgdep2}
\begin{equation*}
\nabla(\log f_0^p)=\nabla^{\mathcal{F}}(\log f_0^p)+\nabla^{\mathcal{G}}(\log g_0^q)=q H_{\mathcal{G}}+p H_{\mathcal{F}},
\end{equation*}
hence, by Corollary \ref{c_critpoint}, $\mathcal{F}$ is a critical point of a functional \eqref{eq_functional}. Analogously, $\mathcal{G}$ is a critical point of a functional \eqref{eq_functional} with a coefficient $q$. Therefore (1) holds. 
\end{proof}

\begin{ackn} This article is based on a part of author's PhD Thesis. The author wishes to thank her advisor Professor Antoni Pierzchalski for helpful discussions. In addition, the author wishes to thank Kamil Niedzia\l omski for helpful discussions that led to improvements of some of the theorems. 
\end{ackn}

\end{document}